\documentclass[11pt,twoside]{article}

\usepackage{mathrsfs,amsfonts,amsmath,amssymb}
\usepackage{latexsym}
\usepackage{comment}

\usepackage[T1]{fontenc}
\usepackage[utf8]{inputenc}
\usepackage{authblk}

\newtheorem{satz}{Theorem}
\newtheorem{proposition}[satz]{Proposition}
\newtheorem{theorem}[satz]{Theorem}
\newtheorem{lemma}[satz]{Lemma}

\newtheorem{corollary}[satz]{Corollary}

\usepackage{listings}
\def\T{\mathsf{T}}

\def\Z{\mathbb {Z}}
\def\F{\mathbb {F}}
\def\E{\mathsf{E}}

\def\l{\lambda}

\def\a{\alpha}

\def\d{\delta}
\def\o{\omega}
\def\({\big (}
\def\){\big )}

\def\le{\leqslant}
\def\ge{\geqslant}
\def\_phi{\varphi}
\def\eps{\varepsilon}

\def\Gr{{\mathbf G}}
\def\FF{\widehat}
\def\k{\kappa}
\def\ov{\overline}

\def\D{\Delta}

\title{On some applications of GCD sums to Arithmetic Combinatorics}
\date{}
\author[]{I.~D.~Shkredov}

\begin{document}
\maketitle
\begin{abstract}
	Using GCD sums, we show that the set of the primes has  small common multiplicative energy with an arbitrary exponentially big integer set  $S$ 
	 and, in particular,  size  of any arithmetic progression in $S$ having  the beginning at zero, is 
	at most $O(\log |S| \cdot \log \log |S|)$. 
	This 
	result 
	can be considered as an integer analogue of Vinogradov's question about the least quadratic non--residue.
	The proof rests on a certain repulsion property of the function $f(x)=\log x$.  
	Also, we consider the case of general $k$--convex functions $f$  and obtain  a new incidence result for collections of the  curves $y=f(x)+c$. 
\end{abstract}

\section{Introduction}

Having a ring $R$ with two operations $+$ and $\cdot$ one can define the {\it sumset} of sets $A,B \subseteq R$ as 
\[
	A+B = \{ a+b ~:~ a\in A,\, b\in B \}
\]
and, similarly,  the {\it product} set
\[
	AB = \{ a\cdot b ~:~ a\in A,\, b\in B \} \,.
\]
The {\it sum--product phenomenon} (see, e.g., \cite{TV}) predicts that additive and multiplicative structure cannot coexist up to some natural algebraic constrains. 
This can be expressed in many different ways see, e.g., \cite{Bourgain_more}
and in our paper we consider just  one of them. 
Let us formulate a particular case of the main result, which is contained in Theorem \ref{t:AP_in_G}
from Section \ref{sec:f=log}.

\begin{theorem}
		Let $S \subset \Z$ be a finite set, $l$ be an integer number, and let $\mathcal{P}^{(l)}$ be the set of primes in the segment $\{1,\dots, l\}$. 
Then 
the condition 
\begin{equation}\label{cond:S_A_new_intr}
	\log |S| =o \left( \frac{l}{\log l} \right)   
\end{equation}
implies 
\begin{equation}\label{f:S_A_new_K_intr}
	|\{ (p,p',s,s') \in \mathcal{P}^{(l)} \times \mathcal{P}^{(l)} \times S \times S ~:~ ps = p's' \}| = o( |\mathcal{P}^{(l)}|^2 |S|) \,.
\end{equation}
In particular, if $|SS| \ll |S|$, then size of any arithmetic progression with the beginning at zero in $S$ does not exceed 
	\begin{equation}\label{f:S_A_new_intr}
	|A| \ll \log |S| \cdot \log \log |S| \,.
	\end{equation}
\label{t:AP_in_G_intr}
\end{theorem}

The result above can be considered as an integer  analogue of Vinogradov's question about the least quadratic non--residue. 
Namely,  having a  number $p$ one can take  the subgroup of squares $\mathcal{R} \subseteq \Z/p\Z$ with the product set $\mathcal{R}\mathcal{R}$ equals $\mathcal{R}$ 
and ask the question about the maximal length of the arithmetic progression with the beginning at zero, belonging to $\mathcal{R}$. 
Size of this arithmetic progression is usually  denoted  as $n_p$ and it is known \cite{Paley2} that there are infinitely many primes such that 
$$
	n_p \gg \log p \cdot \log \log p \,.
$$ 
On the other hand, GRH implies \cite{Ankeny} that $n_p = O(\log^2 p)$ (the best  unconditional bound belongs to Burgess \cite{Burgess} who proved 
$n_p \ll p^{\frac{1}{4\sqrt{e}}+o(1)}$). 
Thus in the integer case our Theorem \ref{t:AP_in_G_intr} gives upper bound \eqref{f:S_A_new_intr} of a comparable quality.

\bigskip

Our another  result is Theorem \ref{t:AS} from Section \ref{sec:f=log}.

\begin{theorem}
	Let $A,S \subset \Z$ be finite sets and $0\le \a <1/6$ be any number.
	Suppose that $|A+A| \le K|A|$ with 
	\begin{equation}\label{cond:AS_1_intr}
	K \ll \exp(\log^{\a} |A|) 
	\end{equation}
	and 
	\begin{equation}\label{cond:AS_2_intr}
	|S| \le \exp\left(\frac{\log^{2-6\a}|A|}{\log \log |A|} \right) \,.
	\end{equation}
	Then for an absolute constant $C>0$ and a certain $a\in A$ one has 
	\begin{equation}\label{f:AS_intr}
	|(A-a)S| \gg |S| \cdot \exp(C \log^{1-3\a} |A|) \,.
	\end{equation}
	In particular, 
	$
	|(A-A)S| \gg |S| \cdot \exp(C \log^{1-3\a} |A|) \,.
	$
	\label{t:AS_intr}
\end{theorem}

The result above 
can be considered as the first step towards the main conjecture from  \cite{BR-NZ} where authors do not assume that the additional condition \eqref{cond:AS_1_intr} takes place (also, see papers \cite{HR-NRI}, \cite{HR-NRII} in this direction).

The method of the proofs of Theorems \ref{t:AP_in_G_intr}, \ref{t:AS_intr} uses so--called GCD sums (see, e.g., \cite{Aisleitner}, \cite{BS}, \cite{Bloom}, \cite{Radziwll}), which are connected with a series of questions of the Uniform Distribution,  as well as Number Theory  in particular, with large values of the zeta function. 
In our paper we follow beautiful exposition of random zeta functions approach from \cite{Radziwll}.  
Thus our method extensively uses the integer arithmetic. 
It is interesting to obtain some analogues of Theorems \ref{t:AP_in_G_intr}, \ref{t:AS_intr} for subsets of $\mathbb{R}$ or $\mathbb{C}$.

If one takes the function  $f(x) = \log x$, then 
Theorem \ref{t:AS_intr} 
can be considered as a repulsion result concerning the logarithmic function.
Namely, estimate \eqref{f:AS_intr} says that $|f(A-a) + \log S|$ must be significantly larger than $|S|$ for rather big sets $S$ as in \eqref{cond:AS_2_intr}.
The first results in the direction were obtained in \cite{HR-NR} for 
general $k$--convex functions (that is having strictly monotone the first $k$ derivatives).
Recall \cite[Theorem 1.4]{HR-NR}.

\begin{theorem}
	Let $A$ be a finite set of real numbers contained in an interval $I$ and let $f$ be a function which is $k$--convex on $I$ for some $k \geq 1$. 
	Suppose that $|A|>10 k$. 
	Then if $|A+A-A| \leq K|A|,$  then we have
	$$
	\left|2^{k} f(A)-(2^{k}-1) f(A)\right| \geq \frac{|A|^{k+1}}{(C K)^{2^{k+1}-k-2}(\log |A|)^{2^{k+2}-k-4}}
	$$
	for some absolute constant $C>0$.	
	\label{t:k-convex_stat_intr}
\end{theorem}

In this direction we obtain a result on common energy of an arbitrary set $S$ and the image of a $k$--convex function (the required definitions can be found in Section \ref{sec:def}). 
Of course general Theorem \ref{t:E_f_intr}  below gives weaker bounds than Theorem \ref{t:AP_in_G_intr} in the particular case $f(x)=\log x$.

\begin{theorem}
	Let $f$ be a function which is $k$--convex on a set $I$ for some $k \geq 1$.
	Suppose that $|I+I-I|\le |I|^{1+\epsilon}$.
	Then for any finite set $S\subset \mathbb{R}$ with  $|I| \ge |S|^\eps$, $\eps \gg 1/k$, $\epsilon \le \exp(-1/(c\eps))$  there is $\delta (\eps)>0$ such that 
	\begin{equation}\label{f:E_f_intr}
	\E^{+} (f(I),S) \ll |I|^2 |S|^{1-\delta(\eps)} \,.
	\end{equation}
	In particular, $|f(I) + S| \gg |S|^{1+\delta(\eps)}$. 
	\label{t:E_f_intr}
\end{theorem}

Using the Pl\"unnecke inequality (see estimate  \eqref{f:Plunnecke} below) one can show that to have growth as in \eqref{f:AS_intr} under the assumptions as in \eqref{cond:AS_1_intr} 
Theorem \ref{t:k-convex_stat_intr} requires the condition 
\begin{equation}\label{f:S_A_old}
	|S| \le \exp(O (\log |A| \cdot \log \log |A|))  
\end{equation}
and our restriction \eqref{cond:AS_2_intr} is wider. 
Theorem \ref{t:AP_in_G}, as well as  Proposition  \ref{p:E(P_z,w)} below require much weaker restrictions on $|S|$ but provide a smaller growth.

Finally, recall the main result from \cite{sh_as}, which can be considered as  a quantitative version of some results from \cite{Bourgain_more}.

\begin{theorem}
	Let $p$ be a primes number, $A, B, C \subseteq \mathbb{F}_{p}$ be arbitrary sets, and $k \geq 1$ be such that $|A||B|^{1+\frac{(k+1)}{2(k+4)} 2^{-k}} \leq p$ and
$$
|B|^{\frac{k}{8}+\frac{1}{2(k+4)}} \geq|A| \cdot C_{*}^{(k+4) / 4} \log ^{k}(|A||B|) \,, 
$$
where $C_{*}>0$ is an absolute constant. Then
$$
\max \{|A B|,|A+C|\} \geq 2^{-3}|A| \cdot \min \left\{|C|,|B|^{\frac{1}{2(k+4)} 2^{-k}}\right\} \,, 
$$
and for any $\alpha \neq 0$
$$
\max \{|A B|,|(A+\alpha) C|\} \geq 2^{-3}|A| \cdot \min \left\{|C|,|B|^{\frac{1}{2(k+4)} 2^{-k}}\right\} \,.
$$
\label{t:sh_as}
\end{theorem}

The result above takes place in $\mathbb{R}$ as well.
In this case we do not need any conditions containing the characteristic $p$.
The main difference between Theorems \ref{t:AS_intr}, \ref{t:E_f_intr} and Theorem \ref{t:sh_as} is that $A$ is large and $B$ is small in Theorem \ref{t:sh_as} but the opposite situation takes place in Theorem \ref{t:E_f_intr} (here $|A|=|I|=|f(I)|$) and similar in Theorem \ref{t:AS_intr}.

\bigskip

The author thanks Christoph  Aistleitner and Sergei Konyagin for useful discussions.

\section{Definitions and preliminaries} 
\label{sec:def}

	Let $\Gr$ be an abelian group.
Put
$\E^{+}(A,B)$ for the {\it common additive energy} of two sets $A,B \subseteq \Gr$
(see, e.g., \cite{TV}), that is, 
$$
\E^{+} (A,B) = |\{ (a_1,a_2,b_1,b_2) \in A\times A \times B \times B ~:~ a_1+b_1 = a_2+b_2 \}| \,.
$$
If $A=B$, then  we simply write $\E^{+} (A)$ instead of $\E^{+} (A,A)$
and the quantity $\E^{+} (A)$ is called the {\it additive energy} in this case. 
Sometimes we write $\E^{+}(f_1,f_2,f_3,f_4)$ for the additive energy of four real functions, namely,
$$
\E^{+}(f_1,f_2,f_3,f_4) = \sum_{x,y,z} f_1 (x) f_2 (y) f_3 (x+z)  f_4 (y+z) \,.
$$
Thus $\E^{+}(f_1,f_2,f_3,f_4)$ pertains to additive quadruples, weighed by the values of $f_1,f_2,f_3,f_4$.
It can be shown using the H\"older inequality (see, e.g., \cite{TV}) that 
\begin{equation}\label{f:E_Ho}
\E^{+}(f_1,f_2,f_3,f_4) \le (\E^{+} (f_1) \E^{+} (f_1) \E^{+} (f_1) \E^{+} (f_1))^{1/4} \,.
\end{equation}
More generally, 
we deal with 
a higher energy
\begin{equation}\label{def:T_k_ab}
\T^{+}_k (A) := |\{ (a_1,\dots,a_k,a'_1,\dots,a'_k) \in A^{2k} ~:~ a_1 + \dots + a_k = a'_1 + \dots + a'_k \}| 
\end{equation}
and similar $\T^{+}_k (f)$  for a general function $f$.  
Sometimes we  use representation function notations like $r_{A+B} (x)$ or $r_{A+A-B}$, which counts the number of ways $x \in \Gr$ can be expressed as a sum $a+b$ or as a sum $a+a'-b$ with $a,a'\in A$, $b\in B$, respectively. 
For example, $|A| = r_{A-A}(0)$ and  $\E^{+} (A) = r_{A+A-A-A}(0)=\sum_x r^2_{A+A} (x) = \sum_x r^2_{A-A} (x)$.  
In the same way define the {\it common multiplicative energy} of two sets $A,B$
$$
\E^{\times} (A,B) =  |\{ (a_1,a_2,b_1,b_2) \in A\times A \times B \times B ~:~ a_1 b_1 = a_2 b_2 \}| \,, 
$$
further $\T^\times_k (A)$, $\T^\times_k (f)$ and so on.

\bigskip 

If $\Gr$ is an abelian group, then the Pl\"unnecke--Ruzsa inequality (see, e.g., \cite{TV}) takes place
\begin{equation}\label{f:Plunnecke}
	|nA-mA| \le \left( \frac{|A+A|}{|A|} \right)^{n+m} \cdot |A| \,,
\end{equation}
and
\begin{equation}\label{f:Plunnecke+}
|nA| \le \left( \frac{|A+A|}{|A|} \right)^{n} \cdot |A| \,.
\end{equation}

Now recall 
our current  knowledge about the Polynomial Freiman--Ruzsa Conjecture, see \cite{Sanders_2A-2A}, \cite{Sanders_3log} and \cite{TV}.
We need a simple consequence of  \cite[Proposition 2.5, Theorem  2.7]{Sanders_3log}.
Recall that if $P_1,\dots, P_d \subset \Z$ are arithmetic progressions, then $Q:=P_1+\dots+P_d$ is a {\it generalized arithmetic progression} (GAP) of {\it dimension} $d$. A generalized arithmetic progression, $Q,$ is called to be {\it proper} if $|Q| = \prod_{j=1}^d |P_j|$. For properties of generalized arithmetic progressions consult, e.g., \cite{TV}.

\begin{theorem}
	Let $A \subset \Z$ be a finite set, $|A+A| \le K|A|$ and $\kappa>3$ be any constant.
	Then there is a proper GAP  $H$ of size at most $|A|\exp (O(\log^\kappa K))$ and dimension $O(\log^\kappa K)$ such that for a set of shifts $X$, $|X| \le \exp (O(\log^\kappa K))$ one has $A\subseteq H+X$. 
\label{t:Sanders_3log}
\end{theorem}

All logarithms are to base $2.$ The signs $\ll$ and $\gg$ are the usual Vinogradov symbols.
For a positive integer $n,$ 
let 
$[n]=\{1,\ldots,n\}.$

\section{The proof of the main result} 
\label{sec:f=log}

Now we obtain Theorem \ref{t:AS_intr} from the Introduction. 
Following the method from \cite{Radziwll} we recall some required definitions.

For each prime $p \in \mathcal{P}$ take a random variable  $X_p$, which is uniformly distributed on $S^{1}$ and let all $X_p$ be independent.  
For every $n \in \mathbb{N}$, $n=p_1^{\o_1} \dots p_s^{\o_s}$, where $p_j \in \mathcal{P}$, $j\in [s]$ are different primes 
put 
$X_n:=\prod_{j=1}^s X_{p_j}^{\o_j}$. 
Then define the random  zeta function by the formula  (let $\a$ be a real number, $\a >\frac{1}{2}$, say)
\begin{equation}\label{def:random_zeta}
	\zeta_{X}(\alpha):=\sum_{n \in \mathbb{N}} \frac{X_n}{n^{\alpha}} = \prod_{p \in \mathcal{P}} \left(1-\frac{X_p}{p^\a} \right)^{-1} \,.
\end{equation}
Using the product formula \eqref{def:random_zeta} one can compute the moments of the random zeta function \eqref{def:random_zeta}, see \cite{Radziwll} (or just similar  calculations in our Lemma \ref{l:moments_new} below).

\begin{lemma}
	Let $l$ be a positive integer.
	Then 
$$
\log \mathbb{E}\left|\zeta_{X}(\alpha)\right|^{2 l} \ll\left\{\begin{array}{ll}
l \log \log l, & \alpha=1 \,, l\ge 3 \\
C (\alpha) l^{1 / \alpha}(\log l)^{-1}, & 1 / 2<\alpha<1 \,, l\ge 3 \\
l^{2} \log \left( \frac{1}{2\a-1}\right), & 1 / 2<\alpha \,, l\ge 1 \,,
\end{array}\right.
$$
where 
$C (\alpha) = \frac{\a}{1-\a} + \frac{\a}{2\a-1}$.
\label{l:Radziwill_moments} 
\end{lemma}

Also, for any function $g: \mathbb{Z} \to \mathbb{C}$ consider the following random analogue of its "multiplicative"\, Fourier transform  
\begin{equation}\label{f:Fourier_random}
	\FF{g} (X) = \sum_{n \in \mathbb{N}} g(n) X_n \,.
\end{equation}
Clearly, we have an analogue of the Parseval identity  
\begin{equation}\label{f:Parseval_random}
	\mathbb{E} |\FF{g} (X)|^2 = \| g\|^2_2 \,,
\end{equation}
and, moreover, for $k\ge 1$ one has 
\begin{equation}\label{f:Tk_random}
	\mathbb{E} |\FF{g} (X)|^{2k} = \T^\times_k (g) \,.
\end{equation}
Further one can compute
$$
	 \mathbb{E} |\FF{g} (X) \zeta_X (\a)|^2 = \sum_{n_1,n_2, m_1,m_2 ~:~ n_1 m_1 = n_2 m_2} \frac{g(m_1) \ov{g} (m_2)}{(n_1 n_2)^\a} 
	 =
$$
\begin{equation}\label{f:gcd}
	 =
	 \zeta (2\a) \sum_{m_1,m_2} g(m_1) \ov{g} (m_2) \cdot \frac{\mathrm{gcd} (m_1,m_2)^{2\a}}{(m_1 m_2)^\a} 
\end{equation} 
and hence GCD sum \eqref{f:gcd} can be interpreted as the multiplicative energy (see the definition of Fourier transform \eqref{f:Fourier_random}) of our weight $g$ with the random zeta function $\zeta_X (\a)$.
It is easy to see (consult estimate \eqref{tmp:25.06_1} below) that it can be converted further to the ordinary multiplicative energy of the function $g$ and the interval $[N]$.

\bigskip  

We follow the method from \cite{Radziwll}, \cite{Bloom}, \cite{Aisleitner} to give the proof of Lemma \ref{l:Radziwill} below.
Generally speaking, our bound \eqref{f:Radziwill} is close to the optimal one, see \cite{BS}.

\begin{lemma}
	Let $w:\Z \setminus\{0\} \to \mathbb{R}^{+}$ be  a non--negative function and $N$ be a positive integer.
	Then for any positive integer $s$ one has 
\begin{equation}\label{f:Radziwill-}
\E^\times ([N],w) \ll N \| w\|_2^2 \exp\left( C\sqrt{s^{-1}\log \log N \cdot \log (\T^\times_{s+1} (w) \|w\|^{-2(s+1)}_2)} + 2\log \log N \right)  
	\ll
\end{equation}
\begin{equation}\label{f:Radziwill}
	\ll  N \| w\|_2^2 \exp\left( C \sqrt{\log \log N \cdot \log (\|w\|_1 \|w\|^{-1}_2)} + 2\log \log N \right) \,,
\end{equation}
	where $C>0$ is an absolute constant. 
\label{l:Radziwill}
\end{lemma}
\begin{proof}
	Let $L=\log N$ and $\a \in (1/2,1]$.  
	Using the Dirichlet principle, as well as estimate \eqref{f:E_Ho}, we find a positive number $U\le N$ such that  
\[
	\E^\times ([N],w) \ll L^2  \sum_{U<n_1,n_2 \le 2U,\, m_1,m_2 ~:~ n_1 m_1 = n_2 m_2} w(m_1) w (m_2) 
		\ll
\]
\[
		\ll 
		L^2 U^{2\a} \sum_{U<n_1,n_2 \le 2U,\, m_1,m_2 ~:~ n_1 m_1 = n_2 m_2} \frac{w(m_1) w (m_2)}{(n_1 n_2)^\a} \,. 
\]
	In terms of the random zeta function \eqref{def:random_zeta}, we see that the last sum is, clearly, 
	does not exceed 
$$
	\sum_{n_1,n_2, m_1,m_2 ~:~ n_1 m_1 = n_2 m_2} \frac{w(m_1) w (m_2)}{(n_1 n_2)^\a}  = \mathbb{E} |\FF{w} (X) \zeta_X (\a)|^2 \,.
$$
	Thus
\begin{equation}\label{tmp:25.06_1}
	\E^\times ([N],w) \ll 	L^2 U^{2\a} \mathbb{E} |\FF{w} (X) \zeta_X (\a)|^2
\end{equation}
	and our task is to estimate the last expectation. 
	Let $l\ge 3$ be an integer parameter, which  we will choose later.
	Also, let $\T_{s+1} = \T^\times_{s+1} (w)$.   
	Thanks to   identities  \eqref{f:Parseval_random}, \eqref{f:Tk_random}, Lemma \ref{l:Radziwill_moments}  and the H\"older inequality, we have 
\begin{equation}\label{f:min_alpha--}
	\mathbb{E} |\FF{w} (X) \zeta_X (\a)|^2 \le  \mathbb{E}^{1-1/l} |\FF{w} (X)|^{2+2/(l-1)} \cdot \mathbb{E}^{1/l} |\zeta_X (\a)|^{2l} 
	\le
\end{equation}
\begin{equation}\label{f:min_alpha-}
	\le
	\left(\mathbb{E} |\FF{w} (X)|^2 \right)^{\frac{s(l-1)-1}{sl}} \left(\mathbb{E} |\FF{w} (X)|^{2s+2} \right)^{\frac{1}{sl}} \cdot \mathbb{E}^{1/l} |\zeta_X (\a)|^{2l} 
=
	\| w\|_2^2 \T^{\frac{1}{sl}}_{s+1}  \|w\|_2^{-\frac{2s+2}{sl}}  \cdot \mathbb{E}^{1/l} |\zeta_X (\a)|^{2l} 
\end{equation}
\begin{equation}\label{f:min_alpha}
	\ll \| w\|_2^2  \exp\left( \frac{1}{ls} \log (\T_{s+1} \|w\|^{-2(s+1)}_2) + \min\left\{ \frac{C(\alpha) l^{1/\a}}{l \log l}, O\left( l\log\frac{1}{2\a-1}\right) \right\} \right)  \,.
\end{equation}
	Put
	$X= s^{-1} \log (\T_{s+1} \|w\|^{-2(s+1)}_2)  \ge 0$.
	First of all, take the second term in the minimum in \eqref{f:min_alpha}.
	In this case we see that the optimal choice of $l$ is $l\sim X^{1/2} \log^{-1/2} (1/(2\a-1))$.
	Hence
\[
	\mathbb{E} |\FF{w} (X) \zeta_X (\a)|^2 \ll \| w\|_2^2  \exp\left( O\left( X^{1/2} \log^{1/2} \frac{1}{2\a-1} \right) \right)  \,.
\]
	Now we take $\a = \frac{1}{2} + \frac{1}{\log N}$ (one can check that this choice of $\a$ allows us to choose  $l\ge 3$) and using $U\le N$, we get in view of \eqref{tmp:25.06_1} that 
\[
	\E^\times ([N],w) \ll L^2 N \| w\|_2^2 \exp\left( O\left(\sqrt{s^{-1}\log \log N \cdot \log (\T_{s+1} \|w\|^{-2(s+1)}_2)} \right) \right) \,.
\]
	To obtain \eqref{f:Radziwill} just notice that $\T_{s+1} \le \|w\|_1^{2s} \| w\|_2^2$. 
	This completes the proof. 
$\hfill\Box$
\end{proof}

\bigskip

Using lemma above we obtain in particular, Theorem \ref{t:AS} from the Introduction.

\begin{theorem}
	Let $A,S \subset \Z$ be finite sets and $0\le \a <1/6$ be any number.
	Suppose that $|A+A| \le K|A|$ with 
\begin{equation}\label{cond:AS_1}
	K \ll \exp(\log^{\a} |A|) 
\end{equation}
	and 
\begin{equation}\label{cond:AS_2}
	|S| \le \exp\left(\frac{\log^{2-6\a}|A|}{\log \log |A|} \right) \,.
\end{equation}
	Then there are at least $\exp(-O(\log^{1-6\a} |A|))$ elements $a\in A$ such that  
\begin{equation}\label{f:AS}
	|(A-a)S| \gg |S| \cdot \exp(O(\log^{1-3\a} |A|)) \,.
\end{equation}
	In addition, if $|S+S| \le K_* |S|$, then \eqref{f:AS} takes place provided 
\begin{equation}\label{cond:AS_2+}
	K_* \log |S| \le \exp\left(\frac{\log^{2-6\a}|A|}{\log \log |A|} \right) \,.
\end{equation}	
\label{t:AS}
\end{theorem}
\begin{proof} 
	Using Theorem \ref{t:Sanders_3log} we find a proper GAP  $H$ of size at most $|A|\exp (O(\log^\kappa K))$ and dimension $d=O(\log^\kappa K)$ such that for a set of shifts $X$, $|X| \le \exp (O(\log^\kappa K))$ one has $A\subseteq H+X$.
	Here $\kappa >3$ is any number.
	We have $H=P_1+ \dots + P_d$, where the sum is direct and all $P_j$ are arithmetic progressions.
	Without loss of generality we can assume that for $P=P_1$ one has $|P| \ge |H|^{1/d}$.
	Also, there is $x\in X$ such that $|A\cap (H+x)| \ge |A|/|X|$ and hence
\[
	|A| \cdot \exp(-O(\log^\kappa K)) \le |A|/|X| \le |A\cap (H+x)| \le \sum_{y\in P_2+\dots+P_d} |A\cap (P+y+x)| \,.
\]
	Thus there exists  $y\in P_2+\dots+P_d +x$ such that 
\[
	|P| \cdot \exp(-O(\log^\kappa K)) \le |P| |A|/|H| \cdot \exp(-O(\log^\kappa K)) 
		=
\]
\begin{equation}\label{tmp:25.09_2}
		= 
		|A| \cdot \exp(-O(\log^\kappa K)) (|P_2| \dots |P_d|)^{-1} \le  |A\cap (P+y)| \,.
\end{equation}
	For any $a\in A\cap (P+y)$, we have $D_*:=A\cap (P+y) - a \subseteq (A-A) \cap (P-P)$. 
	Applying Lemma \ref{l:Radziwill}, the lower bound $|P| \ge |H|^{1/d}$ 
	and using the Holder inequality several times, as well as estimate \eqref{tmp:25.09_2}, we obtain 
\[
	|D_* S| \ge \frac{|D_*|^2 |S|^2}{\E^\times (P-P,S)} \gg  \frac{|A\cap (P+y)|^2 |S|^2}{\E^\times (P,S)} 
		\gg |S| |P| \cdot \exp(-O(\log^\kappa K + \sqrt{\log \log |P| \cdot \log |S|}))
\]
\[
	\gg |S| \cdot \exp\left(O\left(\frac{\log |A|}{\log^\kappa K} - \log^\kappa K - \sqrt{\log \log |A| \cdot \log |S|} \right) \right) \,.
\]
	Thanks to our conditions \eqref{cond:AS_1}, \eqref{cond:AS_2}, we obtain 
\begin{equation}\label{tmp:30.09_1}
	|D_* S| \gg |S| \cdot \exp\left(O\left(\frac{\log |A|}{\log^\kappa K} - \sqrt{\log \log |A| \cdot \log |S|}\right)\right) 
	\gg 
	|S| \cdot\exp\left(O\left(\frac{\log |A|}{\log^\kappa K} \right)\right) 
	\gg 
\end{equation}
\begin{equation}\label{tmp:30.09_2}
	\gg |S| \cdot \exp(O(\log^{1-3\a} |A|))
\end{equation}
as required. 

	To obtain \eqref{cond:AS_2+}  just repeat the previous calculations and use Lemma \ref{l:Radziwill} with the parameter $s=1$. 
	By Solymosi's result \cite{soly} we know that $\E^\times (S) \ll |S+S|^2 \log |S| \ll K^2_* |S|^2 \log |S|$
	and we arrive to an analogue of  \eqref{tmp:30.09_1}, \eqref{tmp:30.09_2}  
\[
	|D_* S| \gg |S| \cdot \exp\left(O\left(\frac{\log |A|}{\log^\kappa K} - \sqrt{\log \log |A| \cdot \log (K_*^2 \log |S|)}\right)\right) \,.
\]
	This completes the proof. 
$\hfill\Box$
\end{proof}

\bigskip

Now consider another zeta function, which allows to make calculations above  better and even simpler. 
Let $\a >0$ be a real number and $z$ be a positive integer.
Then 
\begin{equation}\label{def:random_zeta2}
	\mathcal{Z}_{X}(\alpha):=\prod_{z \le p < 2z} \left(1+\frac{X_p}{p^\a} \right) \,.
\end{equation}
Denote by $\mathcal{P}_z$ the set of all primes in $[z,2z)$ and let $g$ be any non--negative function. 
Since the support of 
$\mathcal{Z}_{X}(\alpha)$
coincides with  all possible products of primes from $\mathcal{P}_z$ and $1$, we see that 
the function $\mathcal{Z}_{X}(\alpha)$ can be used to calculate the common energy of the set $\mathcal{P}_z$ with any function $g$, namely, 
\begin{equation}\label{f:E(g,Z)}
	\E^\times (g,\mathcal{P}_z) < 4^\a z^{2\a} \cdot \mathbb{E} |\FF{g} (X) \mathcal{Z}_{X}(\alpha)|^2 \,.
\end{equation}
Thus to 
compute 
$\E^\times (g,\mathcal{P}_z)$ we need to estimate all moments of the function  $\mathcal{Z}_{X}(\alpha)$ similar to Lemma  \ref{l:Radziwill_moments}.

\begin{lemma}
	Let $\a>0$ be any real number, $l$ be a positive integer and $l\le z^\a$. 
	Then  
\[
	\log \mathbb{E} |\mathcal{Z}_{X}(\alpha)|^{2l} \ll \frac{l^2 z^{1-2\a}}{\log z} \,.
\]
\label{l:moments_new}
\end{lemma}
\begin{proof} 
	In view of the fact that all the variables $X_p$, $p\in \mathcal{P}_z$ are independent,  we have 
\[
	\mathbb{E} |\mathcal{Z}_{X}(\alpha)|^{2l} = \prod_{z \le p < 2z} \mathbb{E} \left(1+\frac{X_p}{p^\a} \right)^l \left(1+\frac{\ov{X}_p}{p^\a} \right)^{l}
	:= 
	\prod_{z \le p < 2z} E_{l} (p) \,, 
\]
	and 
\[
	E_{l} (p) = \frac{1}{2\pi} \int_{0}^{2\pi} \left(1+\frac{e^{i\theta}}{p^\a} \right)^l \left(1+\frac{e^{-i\theta}}{p^\a} \right)^{l}\, d\theta 
	=
	\sum_{n=0}^{l} \binom{l}{n}^2 \frac{1}{p^{2\a n}} \,.
\]
	Using the condition $l\le z^\a$, we obtain $\log E_{l} (p) \le 2l^2/p^{2\a}$.  
	Hence
\[
	\log \mathbb{E} |\mathcal{Z}_{X}(\alpha)|^{2l} \le 2l^2 \sum_{z \le p < 2z} p^{-2\a} \ll \frac{l^2 z^{1-2\a}}{\log z} \,.
\] 
	This completes the proof. 
$\hfill\Box$
\end{proof}

\bigskip 

Now we formulate an analogue of Lemma \ref{l:Radziwill} allowing to calculate the  common energy of the set $\mathcal{P}_z$ with a general weight $w$.

\begin{proposition}
	Let $w:\Z \setminus\{0\} \to \mathbb{R}^{+}$ be  a non--negative function and $s,z$ be positive integers.
	Suppose that 
\begin{equation}\label{cond:l}
	\log (\T^\times_{s+1} (w) \|w\|^{-2(s+1)}_2) \le \frac{s z}{\log z} \,.
\end{equation}
Then for any $\a>0$ the following holds 
\begin{equation}\label{f:Radziwill-_new}
	\E^\times (\mathcal{P}_z,w) \ll z^{2\a} \| w\|_2^2 \exp\left( C z^{1/2-\a}   \sqrt{s^{-1} \log^{-1} z \cdot \log (\T^\times_{s+1} (w) \|w\|^{-2(s+1)}_2)} \right) \,.
\end{equation}
	In particular, for any $\eps>0$ one has 
\begin{equation}\label{f:Radziwill-_new_eps}
\E^\times (\mathcal{P}_z,w) \ll (\eps z)^{2} \| w\|_2^2 \exp\left( C \eps^{-1} z^{-1/2}   \sqrt{s^{-1} \log^{-1} z \cdot \log (\T^\times_{s+1} (w) \|w\|^{-2(s+1)}_2) } 
 \right) \,.
\end{equation}
\label{p:E(P_z,w)}
\end{proposition}
\begin{proof} 
	Let 
	$X=s^{-1}  \log (\T^\times_{s+1} (w) \|w\|^{-2(s+1)}_2)$.
	Choose  $l=(X\log z /z)^{1/2} z^\a$. 
	Thanks to our assumption \eqref{cond:l}, we have $l \le z^\a$. 
	Using Lemma \ref{l:moments_new} as in lines \eqref{f:min_alpha--}---\eqref{f:min_alpha}, combining with  bound \eqref{f:E(g,Z)}, we get 
\[
\E^\times (\mathcal{P}_z,w)  \ll z^{2\a} \| w\|_2^2  \exp\left( \frac{1}{ls} \log (\T^\times_{s+1} (w) \|w\|^{-2(s+1)}_2) + \frac{l z^{1-2\a}}{\log z}  \right)
\ll
\]
\begin{equation}\label{f:min_alpha_new}  
\ll 
	z^{2\a} \| w\|_2^2  \exp\left( C z^{1/2-\a}   \sqrt{s^{-1} \log^{-1} z \cdot \log (\T^\times_{s+1} (w) \|w\|^{-2(s+1)}_2)} \right) \,.
\end{equation}
	Taking $\a = 1 - \frac{\log (1/\eps)}{\log z}$, we obtain  
\begin{equation}\label{f:min_alpha_new'}  
	\E^\times (\mathcal{P}_z,w)  \ll  (\eps z)^{2} \|w\|_2^2 \exp\left( C \eps^{-1} z^{-1/2}   \sqrt{s^{-1} \log^{-1} z \cdot \log (\T^\times_{s+1} (w) \|w\|^{-2(s+1)}_2) } \right) 
\end{equation}
	as required. 
$\hfill\Box$
\end{proof}

\bigskip 

Now we derive an upper bound for the common energy of the set of the primes in a segment and an arbitrary set.
It shows that in a sense the primes  "repulse"\, the other sets.  

\begin{theorem}
	Let $S \subset \Z$ be a set, $l$ be an integer number, and $\mathcal{P}^{(l)} := [l] \cap \mathcal{P}$. 
	Then for any $d\neq 0$ the condition 
	\begin{equation}\label{cond:S_A_new}
	\log |S| \ll \frac{\eps^{} l}{\log l}    
	\end{equation}
	implies 
	\begin{equation}\label{f:S_A_new_K}
		\E^{\times} (d \cdot \mathcal{P}^{(l)},S) \le \eps |\mathcal{P}_l|^2 |S| \,.
	\end{equation}
\label{t:AP_in_G}
\end{theorem}
\begin{proof} 
	Take any $z\le [l/2]$. 
	By estimate \eqref{f:Radziwill-_new} of Proposition \ref{p:E(P_z,w)}, we get for any $\eps_* >0$ 
\[
	\E^\times (\mathcal{P}_z,S) \ll (\eps_* z)^{2}  |S| \exp\left( C \eps^{-1}_* z^{-1/2} \sqrt{ \log^{-1} z \cdot \log |S| } \right) \,.
\]	
	Summing over $z$ of the form $2^j\le l/2$, 
	 we obtain 
\[
	\E^{\times} (d \cdot \mathcal{P}^{(l)},S) \ll \eps^2_* l^2 |S|  \exp\left( C \eps^{-1}_* l^{-1/2}   \sqrt{ \log^{-1} l \cdot \log |S| } \right) \,.
\]
	Now put $\eps^2_*  = \frac{\eps}{\log^2 l}$. 
	Then thanks to our assumption \eqref{cond:S_A_new}, we have \eqref{f:S_A_new_K}. 
%
%
	This completes the proof. 
	$\hfill\Box$
\end{proof}

\bigskip

Of course in Theorem \ref{t:AP_in_G} one can consider more general arithmetic  progressions as well but in this case one should control the beginning and the step of such progression, simultaneously.      

\section{On general  $k$--convex functions} 
\label{sec:k-convex}

In \cite[Theorem 1.3]{HR-NR} authors obtained the following growth result for sequences of the form $A = f([N])$, where $f$ is  an arbitrary $k$--convex function.

\begin{theorem}
	Let $k\ge 2$ be an integer and let $A$ be a $k$--convex sequence. 
	Then
	\[
	|2^k A - (2^k-1) A| \gg \frac{|A|^{k+1}}{2^{k^2}} \,.
	\]
	\label{t:k-convex_intr}
\end{theorem}


Thus Theorem \ref{t:k-convex_stat_intr} from the Introduction can be considered as  a "statistical"\, version of Theorem \ref{t:k-convex_intr}. 
Also, notice that the dependence on $k$ in Theorem \ref{t:k-convex_intr} is better.

In this Section we show how Theorem \ref{t:k-convex_stat_intr} implies an upper bound for the higher energy of any $k$--convex function. 
Basically, we repeat the combination of the arguments from \cite[Theorem 13]{RS} and \cite[Theorem 23]{s_energy}. 

\begin{theorem}
	Let $f$ be a function which is $k$--convex on a set $I$ for some $k \geq 1$. 
	Suppose that $|I+I-I|\le |I|^{1+\epsilon}$.  
	Then for all $l\le 2^k$, $\epsilon \le \frac{\log l}{l}$ one has  
	\begin{equation}\label{f:T_l}
	\T^{+}_{2^l} (f(I)) \ll |I|^{2^{l+1} - c \log l}  \,.
	\end{equation}
	for a certain absolute constant $c>0$.
	\label{t:T_l}	
\end{theorem}
\begin{proof}
	Put $A=f(I)$.
	Let $\T_{2^j} := \T^{+}_{2^j}(A)$ and $\T_1 = |A|^2$.
	Our task is to prove for any $j\in [l]$ that 
	\begin{eqnarray}\label{f:ind}	
	\T_{2^j} \le \frac{\T_{2^{j-1}} |A|^{2^j}}{Q} \,, 
	\end{eqnarray}
	where $Q=|A|^{\frac{c \log j}{j}}$ because it clearly implies \eqref{f:T_l}. 
	Suppose not. 
	Put $L=O(k \log |A|)$. 
	By the dyadic Dirichlet principle and the H\"older inequality in the form \eqref{f:E_Ho} there is a number $\D >0$ and a set $P = \{ x \in \Z ~:~ \D <  r_{2^{j-1}A} (x) \le 2 \D \}$ such that
	\begin{equation}\label{tmp:19.09_3}
	L^4 \D^4 \E^{+}(P) \ge \T_{2^j} \ge \frac{|A|^{2^j} \T_{2^{j-1}}}{Q} \ge \frac{(\D |P|)^2 \D^2 |P|}{Q} \,.
	\end{equation}
	Indeed, we can assume that \eqref{f:T_l} does not hold (otherwise there is nothing to prove) and 
	thus  by our condition $j\le l\le 2^k$ one has  
	\[
	|A|^{2^{j-1}-1} \ge \D \gg |A|^{2^{j-1}-c \log j} 
	\]
	and hence we do indeed have the upper bound \eqref{tmp:19.09_3} with the quantity $L$.
	Further from \eqref{tmp:19.09_3}, we obtain
	$\D \ge L^{-4} \T_{2^j} |A|^{-3\cdot 2^{j-1}}$ and
	\[
	\E^{+}(P) \gg L^{-4} \frac{|P|^3}{Q} := \frac{|P|^3}{Q_1} \,.
	\]
	Also notice that $\D^4 \E^{+}(P) \le \D^2 |P| (\D |P|)^2 \le \T_{2^{j-1}} (\D |P|)^2$ and hence from \eqref{tmp:19.09_3}, we get 
	\begin{equation}\label{tmp:19.09_4}
	\D |P| \ge \frac{|A|^{2^{j-1}}}{L^2 Q^{1/2}} \,.
	\end{equation}
	Similarly, $\D^4 \E^{+}(P) \le (\D^2 |P|) |A|^{2^j-2} |P|^2 \le \T_{2^{j-1}} |A|^{2^j-2} |P|^2$ and thus from \eqref{tmp:19.09_3}, we derive
	\begin{equation}\label{tmp:19.09_4'}
	|P| \ge \frac{|A|}{L^2 Q^{1/2}} \,.
	\end{equation}
	By the   Balog--Szemer\'edi--Gowers Theorem (see, e.g., \cite{TV}), 
	we find $P_* \subseteq P$ such that $|P_*| \gg |P| Q_1^{-C_*}$, and  $|P_*+P_*| \ll Q^{C_*}_1 |P_*|$.
	Here $C_* > 1$ is an absolute constant, which may change from line to line.
	By the definition of the set $P$, we have 
	\[
	\D |P_*| \le \sum_{x\in P_*} r_{2^{j-1}A} (x) = \sum_{x_1,\dots, x_{2^{j-1}-1}\in A} r_{P_*-A} (x_1 +\dots +x_{2^{j-1}-1}) \,.
	\]
	Hence there is a shift $x$ and a set $A_* \subseteq A \cap (P_*-x)$ such that 
	\begin{equation}\label{tmp:16.09_1}
	|A_*| \ge \D |P_*| / |A|^{2^{j-1}-1} \gg |A| (LQ)^{-C_*} \,.
	\end{equation}
	Here we have used bound \eqref{tmp:19.09_4}.
	The set $A_*$ has the form $A_* = f(S)$, where $S\subseteq I$ is a set of the same size. 
	Clearly, 
	$$
	|S+S-S|\le |I+I-I| \le |I|^{1+\epsilon} = |A|^{1+\epsilon}/|A_*| \cdot |S| := K|S| \,.
	$$
	Applying Theorem \ref{t:k-convex_stat_intr} with a parameter $t = t(j) \le k$, which we will choose later, combining with  inequality \eqref{f:Plunnecke}, we obtain 
	\[
	\frac{|A_*|^{t+1}}{(C K)^{2^{t+1}-t-2}(\log |A_*|)^{2^{t+2}-t-4}} \le |2^{t} A_* - (2^{t}-1)A_*| \le |2^{t} P_* - (2^{t}-1) P_*| 
	\ll
	\]
	\begin{equation}\label{tmp:03.10_1}
	\ll
	Q^{(2^{t+1}-1)C_*}_1 |P_*| \,. 
	\end{equation}
	Thanks to estimate \eqref{tmp:16.09_1}, we know that $K \ll (LQ)^{C_*} |A|^\epsilon$.
	By the assumption $\epsilon \le \frac{\log l}{l}$ and hence  $K \ll (LQ)^{C_*}$ (with another constant $C_*$ of course) by our choice of $Q$.  
	Using this 
	estimate, 
	as well as both inequalities  from \eqref{tmp:16.09_1}, combining with \eqref{tmp:19.09_4} and the lower bound $|P_*| \gg |P| Q_1^{-C_*}$, we derive from \eqref{tmp:03.10_1} 
	\[
	\D |P_*| \cdot |A|^{t+1-2^{j-1}} Q^{-C_* 2^t}_1  \le \left( \frac{\D |P_*|}{A|^{2^{j-1}-1}}  \right)^{t+1} \ll Q^{C_* 2^t}_1 |P_*| \,.
	\]
	Hence 
	$$
	\D |A|^{t+1-2^{j-1}} \ll Q^{C_* 2^t}_1 
	$$
	and in view of \eqref{tmp:19.09_3}, we get
	\[
	|A|^{2^{j-1}- (t+1)} Q^{C_* 2^t}_1 |A|^{3\cdot 2^{j-1}} \ge \T_{2^j} \ge |A|^{2^{j+1} - c \log j} \,.
	\]
	Now take the parameter $t$ as $t(j) = \log j$. 
	It follows that for sufficiently large constant $C'$ we get  $Q\gg |A|^{\frac{\log j}{C' j}}$. 
	This completes the proof. 
	$\hfill\Box$
\end{proof}

\bigskip 

Theorem \ref{t:T_l}	 can be used to obtain a series of lower bounds for various combinations of {\it different} sets see, e.g., \cite[Corollary 1.5]{HR-NR}.
We restrict ourself  by just one consequence.  
Much more stronger results for subsets of $\Z$ were obtained in \cite{HR-NRI}, \cite{HR-NRII}. 

\begin{corollary}
	Let $m$ be a positive integer, $A_1,\dots, A_{2^m} \subset \mathbb{R}$ be sets of the same size $|A_1|$, $|A_j A_j| \ll |A_j|$, $j\in [2^m]$.  
	Then for any non--zero shifts $z_1,\dots,z_{2^m}$ one has 
	\[
	|(A_1+z_1) \dots (A_{2^m}+z_{2^m})| \gg |A_1|^{c \log m} \,. 
	\]
\end{corollary}
\begin{proof} 
	For any $z\neq 0$ consider the function $f_z (x) = \log (z+e^x)$. 
	Then $f_z$ is $k$--convex for any $k$. 
	Also, for $I=\log A$, where $A$ is any of the sets $A_j$, $j\in [2^m]$ one has in view of \eqref{f:Plunnecke} that $|I+I-I|\ll |I|$. 
	Applying Theorem \ref{t:T_l} for  $f=f_z$, and $l=m$, we see that $\T^\times_{2^m} (A+z) \ll |A|^{2^{m+1}- c \log m}$.
	Hence by the H\"older inequality
	\[
	|A_1|^{2^{m+1}} \le |(A_1+z_1) \dots (A_{2^m}+z_{2^m})| \cdot \sum_x r^2_{(A_1+z_1) \dots (A_{2^m}+z_{2^m})} (x) 
	\le 
	\]
	\[
	\le 
		|(A_1+z_1) \dots (A_{2^m}+z_{2^m})| \cdot \left(\prod_{j=1}^{2^m} \T^\times_{2^m} (A_j+z_j) \right)^{1/2^m}
	\ll
	|(A_1+z_1) \dots (A_{2^m}+z_{2^m})| \cdot |A_1|^{2^{m+1}- c \log m}
	\]
	as required. 
	$\hfill\Box$
\end{proof}

\bigskip 

Now we obtain a new incidence result for one--parametric curves. 

\begin{theorem}
	Let $f$ be a function which is $k$--convex on a set $I$ for some $k \geq 1$.
	Suppose that $|I+I-I|\le |I|^{1+\epsilon}$ and $\epsilon \le \frac{\log k}{k}$.  
	Then for any finite sets $B,C\subset \mathbb{R}$ with  $|I| \ge |B|^\eps$, $\eps \gg 1/k$ and $\epsilon \le \exp(-1/(c\eps))$  
	there is $\delta (\eps) \ge \exp(-\exp (O(1/\eps)))>0$ such that 
	\begin{equation}\label{f:inc_f}
	|\{ (i,b,c)\in I\times B \times C ~:~ f(i) + b = c \}| \ll \sqrt{|B||C|} |I| \cdot |B|^{-\d(\eps)} \,.
	\end{equation}
	\label{t:inc_f}
\end{theorem}
\begin{proof} 
	Put $A=f(I)\cup (-f(I))$ and let $\sigma$ be cardinality of the set on the left--hand side of \eqref{f:inc_f}. 
	Using the Cauchy--Schwarz inequality several times, we obtain for any $j$ 
	\[
	\sigma^{2^j} \le |C|^{2^{j-1}} |B|^{2^{j-1}-1} \sum_x r_{2^j A} (x) r_{B-B} (x) \,.
	\]
	Applying the Cauchy--Schwarz  inequality one more time, we get
	\[
	\sigma^{2^{j+1}} \le |C|^{2^{j}} |B|^{2^{j}-2} \E^{+}(B) \T_{2^{j}} (A) \,. 
	\]
	Now suppose that $j\le 2^k$.
	Then by Theorem \ref{t:T_l} and the trivial bound $\E^{+}(B)\le |B|^3$, we obtain 
	\[
	\sigma^{2^{j+1}} \ll |C|^{2^{j}} |B|^{2^{j}} \cdot |B|  |I|^{2^{j+1} - c\log j} \,.
	\]
	It gives us
	\[
	\sigma \ll \sqrt{|B||C|} |I| \cdot \left( \frac{|B|}{|I|^{c\log j}} \right)^{2^{-(j+1)}}
	\]
	By our assumption $|I| \ge |B|^\eps$ and hence taking $j \gg \exp(1/(c\eps))$, we derive 
	\[
	\sigma \ll \sqrt{|B||C|} |I| \cdot |B|^{-2^{-(j+1)}}
	\]
	as required. 
	Here $\delta(\eps) \sim \exp(-\exp (1/c\eps))$. 
	This completes the proof. 
	$\hfill\Box$
\end{proof}

\bigskip 

The incidence result above implies Theorem \ref{t:E_f_intr} from the Introduction.

\begin{corollary}
	Let $f$ be a function which is $k$--convex on a set $I$ for some $k \geq 1$.
	Suppose that $|I+I-I|\le |I|^{1+\epsilon}$.
	Then for any finite set $B\subset \mathbb{R}$ with  $|I| \ge |B|^\eps$, $\eps \gg 1/k$, $\epsilon \le \exp(-1/(c\eps))$  there is $\delta (\eps)>0$ such that 
	\begin{equation}\label{f:E_f}
	\E^{+} (f(I),B) \ll |I|^2 |B|^{1-\delta(\eps)} \,.
	\end{equation}
	In particular, $|f(I) + B| \gg |B|^{1+\delta(\eps)}$. 
	\label{c:E_f}
\end{corollary}
\begin{proof} 
	Let $\tau>0$ be  a real number and 
	$$S_\tau = \{ s\in \mathbb{R} ~:~ |\{ (i,b) \in I \times B ~:~ f(i)+b = s\}| \ge \tau \} \,.$$ 
	Using Theorem \ref{t:T_l}, we have 
	\[
	\tau |S_\tau| \le |\{ (i,b,s) \in I \times B \times S_\tau ~:~ f(i)+b = s\}| \ll \sqrt{|B| |S_\tau|} |B|^{-\delta(\eps)} \,. 
	\]
	By summation we obtain \eqref{f:E_f} and  the bound $|f(I) + B| \gg |B|^{1+\delta(\eps)}$ follows from the Cauchy--Schwarz inequality. 
	This completes the proof. 
	$\hfill\Box$
\end{proof}

\bigskip

\noindent{Steklov Mathematical Institute,\\
	ul. Gubkina, 8, Moscow, Russia, 119991}
\\
and
\\
IITP RAS,  \\
Bolshoy Karetny per. 19, Moscow, Russia, 127994\\
and 
\\
MIPT, \\ 
Institutskii per. 9, Dolgoprudnii, Russia, 141701\\
{\tt ilya.shkredov@gmail.com}

\end{document}